\documentclass{amsart}[12]
\usepackage{amsmath}
\usepackage{graphicx}
\usepackage{latexsym}
\usepackage{color}
\usepackage{amscd}
\usepackage[all]{xy}
\usepackage{enumerate}
\usepackage{pinlabel}
\usepackage{hyperref}
\newtheorem{thm}{Theorem}

\newtheorem{cor}[thm]{Corollary}
\newtheorem{prop}[thm]{Proposition}

\newtheorem{theorem}[thm]{Theorem}
\newtheorem{lemma}[thm]{Lemma}

\theoremstyle{definition}

\newtheorem{remark}[thm]{Remark}

\newtheorem*{question}{Question}

\newcommand{\CPb}{\overline{\mathbb{CP}}{}^{2}}
\newcommand{\CP}{{\mathbb{CP}}{}^{2}}

\newcommand{\RP}{{\mathbb{RP}}{}^{2}}

\newcommand{\K}{\textit{KB}}

\def \x {\times}

\begin{document}

\title[Knotted surfaces in $4$-manifolds and stabilizations]
{Knotted surfaces in $4$-manifolds \\and stabilizations}

\author[R. \.{I}. Baykur]{R. \.{I}nan\c{c} Baykur}
\address{Department of Mathematics and Statistics, University of Massachusetts, Amherst, MA 01003}
\email{baykur@math.umass.edu}

\author[N. Sunukjian]{Nathan Sunukjian}
\address{Department of Mathematics and Statistics, University of Massachusetts, Amherst, MA 01003}
\email{nsunukjian@math.umass.edu}

\begin{abstract}
In this paper, we study stable equivalence of exotically knotted surfaces in $4$-manifolds, surfaces that are topologically isotopic but not smoothly isotopic. We prove that any pair of embedded surfaces in the same homology class become smoothly isotopic after stabilizing them by handle additions in the ambient $4$-manifold, which can moreover assumed to be attached in a standard way (locally and unknottedly)  in many favorable situations. In particular, any exotically knotted pair of surfaces with cyclic fundamental group complements become smoothly isotopic after a same number of standard stabilizations --analogous to C.T.C. Wall's celebrated result on the stable equivalence of simply-connected $4$-manifolds. We moreover show that all constructions of exotic knottings of surfaces we are aware of, which display a good variety of techniques and ideas, produce surfaces that become smoothly isotopic after a single stabilization.
\end{abstract}

\maketitle

\vspace{0.2in}
\section{Introduction}

A pair of embedded surfaces in a $4$-manifold are said to be \emph{exotically knotted} if they are topologically isotopic, but not smoothly isotopic. Since topologically isotopic surfaces have the same topology (orientability and genus) and represent the same homology class, exotic knottings are more rigid examples of \emph{exotic embeddings} of surfaces, where the surfaces are ambiently homeomorphic but not diffeomorphic. 

Since the advent of gauge theory, many infinite families of exotically knotted and exotically embedded surfaces in closed oriented $4$-manifolds have been produced; see e.g. \cite{AKMR, FKV, Finashin1, Finashin2, Finashin3, FintushelStern1, FintushelStern2, HS, Kim, KimRuberman1, KimRuberman2, KimRuberman3, Mark}. Notably, the phenomenon of \textit{infinite} exotic knottings is unique to dimension $4$. These diverse examples owe much to the ingenious construction techniques introduced by many authors, which make it possible to simultaneously control the underlying algebraic topology so as to invoke Freedman's theory, and control the smooth topology to be able to calculate Donaldson and Seiberg-Witten invariants in order to obstruct smooth equivalence. Prompted by the wealth of these examples, our goal in this article is to develop an appropriate notion of \textit{stable equivalence for knotted surfaces} and analyze exotic knottings under this equivalence. By a \emph{stabilization} of an embedded surface, we will simply mean attaching an embedded handle to the surface, increasing its genus. A \emph{standard stabilization} will then refer to adding an unknotted handle attached locally (i.e. by taking internal connected sum with a trivial $2$-torus in a small ball neighborhood of a point on the surface). We have the following general result:

\begin{thm} \label{mainthm1}
Any pair of homologous embedded surfaces $\Sigma_i$, $i=0,1$ in a compact oriented $4$-manifold $X$ become smoothly isotopic after enough stabilizations. If 
$\pi_1(\partial \, (\nu \Sigma_i) )$ surjects on to $\pi_1(X \setminus \Sigma_i)$ under the inclusion homomorphism, then these handles can be attached in a standard way. In particular, if $\Sigma_i$ are exotically knotted surfaces in $X$ with cyclic fundamental group complements, then they become smoothly isotopic after the same number of standard stabilizations. 
\end{thm}

\noindent In the case of knotted $2$-spheres in $S^4$, the first result is due to Hosokawa and Kawauchi \cite{HK}. Our proof of this general theorem rests on rather classical differential topology and surgery theory arguments, drawing ideas from \cite{B, GL, HK, K, K2, Kirby, Mas}, whereas the notion of stabilization we arrive at is well suited for studying the modern theory of exotically knotted surfaces.

Our theorem is analogous to C.T.C. Wall's classical result \cite{Wall}, which states that any pair of homeomorphic closed simply-connected oriented $4$-manifolds $X_i$, $i=0,1$, become diffeomorphic after stabilizing by taking connected sums with some number of $S^2 \x S^2$'s. Indeed, this is more than an analogy. When double branched covers along exotically knotted $\Sigma_i$ in $X$ result in an exotic pair of \linebreak $4$-manifolds $\widetilde{X}_i$ (as it is the way to argue that they are smoothly knotted in many examples in the literature, e.g. \cite{Finashin3, FKV}) there is a direct connection. In this case, our stabilization of the embedded surfaces amounts to taking relative connect sum of $(X, \Sigma_i)$ with $(S^4, T^2$), the standard unknotted embedding of $T^2$ in $S^4$, so the double cover along the stabilized surface then gives a connected sum of $\widetilde{X}_i$ with $S^2 \x S^2$, (which is the double branched cover of $S^4$ along the unknotted $T^2$).  Similarly, a twisted stabilization of $\Sigma_i$ yields a stabilization of the double cover $\widetilde{X}_i$ with $S^2 \tilde{\x} S^2$, which is the double branched cover of $S^4$ along the unknotted Klein bottle.

More than $50$ years after Wall, it is still unknown if a single stabilization suffices to get a diffeomorphism between any pair of homeomorphic closed simply-connected oriented 4-manifolds; this is one of the fundamental open problems on $4$-manifolds. Exotic $4$-manifolds of course provide examples where at least one stabilization is necessary, whereas in \cite{BS}, we showed that all construction methods employed up to date to generate infinite families of exotic $4$-manifolds always yield 4-manifolds which become diffeomorphic after a single stabilization.\footnote{Well, almost: we needed an extra blow-up to deal with the spin case, which we expect to be superfluous, and can possibly be avoided by carefully keeping track of involved framings. Another way to avoid the extra blow-up would be using twisted stabilizations with $S^2 \tilde{\x} S^2= \CP \# \CPb$ instead, as in the work of Auckly \cite{Auckly}. } We will examine the analogous question for surface stabilizations: how many are needed to make exotic knottings smoothly isotopic? We analyze in detail various constructions of exotically knotted and exotically embedded surfaces\footnote{Not all exotic embeddings are exotic knottings; one can for instance have ambiently homeomorphic $\Sigma_i$ representing different homology classes \cite{BaykurHayano}. On the other hand, many examples of exotic embeddings in the literature, such as the examples of $\Sigma_i$ in $S^4$ constructed in \cite{Finashin3, FKV}, can be seen to be topologically isotopic \cite{Sunukjian}, and are thus exotically knotted.}, which produced non-orientable surfaces in $S^4$ in the pioneering works of Finashin, Kreck and Viro \cite{FKV} and Finashin \cite{Finashin3}, and orientable surfaces (often topologically isotopic to symplectic/complex curves) in many other $4$-manifolds in the works of Fintushel and Stern \cite{FintushelStern1, FintushelStern2}, Kim and Ruberman \cite{Kim, KimRuberman1, KimRuberman2, KimRuberman3}, Finashin \cite{Finashin1, Finashin2}, Mark \cite{Mark}, Hoffman and Sunukjian \cite{HS}. 

We prove that all these different constructions and their immediate generalizations are subject to the same end result: examples produced using these techniques become smoothly isotopic after a single stabilization. A portion of our detailed analysis can be summarized as follows:

\begin{thm} \label{mainthm2}
Any pair of exotically knotted surfaces $\Sigma_i$, $i=0,1$, in a compact oriented $4$-manifold $X$ obtained through rim surgery \cite{FintushelStern1}, twist rim surgery \cite{Kim, KimRuberman1}, annulus surgery \cite{FKV}, or any other kind of tangle surgery \cite{Finashin2,HS}, become smoothly isotopic after a single standard stabilization. 
\end{thm}

There is yet another source of exotic knottings of surfaces, coming from dissolving exotic $4$-manifolds. Recall that a simply-connected $X$ is said to be \emph{almost completely decomposable} (ACD), if $X \# \CP$ smoothly dissolves into a connected sum of standard $4$-manifolds $S^4$, $S^2 \x S^2$, $\CP$, and $\CPb$. For any exotic pair of almost completely decomposable $X_i$, inclusions of $\mathbb{CP}^1 \subset \CP$ into the standard manifold $X= X_i \# \CP$ make up a pair of exotically knotted $2$-spheres $\Sigma_i$ in $X$. Recently in \cite{AKMR}, Auckly, Kim, Melvin and Ruberman explored this recipe (also see Akbulut's note \cite{Akbulut}) to produce exotically knotted $2$-spheres $\Sigma_i$ which become smoothly isotopic in the $4$-manifold $X \# S^2 \x S^2= X_i \# \CP \# S^2 \x S^2$ (i.e. after \textit{stabilizing the complements} of $X \setminus \Sigma_i$ by taking connected sum with $S^2 \x S^2$ a l\'{a} Wall). Our analysis extends to their examples: after a single standard handle attachment to $\Sigma_i$ in $X$ as in Theorem~\ref{mainthm1} we again get smoothly isotopic surfaces. This is possible because their examples are very local in nature. However, since in general the ACD property is very much global as opposed to the local nature of all other aforementioned methods of smooth knottings, one might suspect that this framework (or a similar one which can be formulated by dissolving $X_i$ after stabilizing with $S^2 \x S^2$ instead) could potentially produce examples exhibiting more exotic behavior than all the others we have covered in this paper. This further ties the question on the complexity of stable equivalences of exotic $4$-manifolds to that of stable equivalences of exotically knotted surfaces in $4$-manifolds, which in some ways might be more tractable:

\begin{question}
Does every pair of topologically isotopic surfaces $\Sigma_i$, $i=0,1$, in a closed simply-connected oriented $4$-manifold $X$ become smoothly isotopic after a single standard handle attachment? 
\end{question}

\smallskip
The organization of our paper is as follows: In Section~2, we will build the foundation by describing exactly what we mean when we say that we will stabilize a surface by adding a handle to it in the ambient $4$-manifold, specifically dealing with the differences between orientable and non-orientable surfaces, and discussing issues regarding the framings of these handles. We then prove the results summarized in Theorem~\ref{mainthm1}. In Section~3, we analyze all of the different constructions of exotic knottings of surfaces discussed above, and show how these techniques, and their immediate generalizations, are destined to result in examples of knotted surfaces which become smoothly isotopic after a single standard stabilization.

\vspace{0.3in}
\noindent \textit{Acknowledgements.} We thank Seiichi Kamada for his comments on a draft of this paper. The first author was partially supported by the Simons Foundation Grant $317732$. 

\newpage

\section{Stable equivalence for surfaces in $4$-manifolds}

The goal of this section is to develop the right notion of stable equivalence for knotted surfaces in $4$-manifolds, based on a very natural operation: increasing the genus of the surface.  We will first first discuss how this can be done ambiently as a handle attachment, how to control the framings (importantly, the orientability of the surface) and the circumstances when one can perform it in a standard way (i.e. without using knotted handles).  

\vspace{0.1in}
\subsection{Adding handles to an embedded surface} \

When we talk about \textit{adding a handle} to a surface $\Sigma$ in $X$, by a \textit{handle}, we will mean a $3$-dimensional $1$-handle $h$ embedded in the $4$-manifold which intersects the surface only along its attaching region. The new surface $\Sigma'$, the surface \textit{with the added handle}, will be the result of cutting out this intersection, and gluing in the portion of the boundary of $h$ complementary to the attaching region. (As usual, this can be thought of as a cobordism $\Sigma \x [0,1] \cup h$ from $\Sigma$ to $\Sigma'$.) The data necessary to specify how $h$ is attached to $\Sigma$ is spelled out in the lemma below. We will call this procedure \textit{stabilization} of $\Sigma$. 

A handle is called a \textit{trivial handle} if it is isotopic to a handle attached in a small ball neighborhood of a point on $\Sigma$. Up to isotopy there are two such handle attachments; the \textit{untwisted} one that respects the orientation of the surface, and the \textit{twisted} one that doesn't. Note that this depends on the framing of the core arc of $h$. When $\Sigma$ is non-orientable, there are still two ways to attach a handle trivially, which can be distinguished by assigning local orientations. One can think of a trivial handle as specifying the surface $(X, \Sigma')= (X, \Sigma) \# (S^4,T^2)$ in the untwisted case, and $(X, \Sigma')= (X, \Sigma) \# (S^4, \K)$, in the twisted case. Here the $2$-torus $T^2$ and the Klein bottle $\K$ are understood to be embedded in $S^4$ in a trivial way, meaning they bound a solid handlebody in $S^4$. We will refer to attaching a trivial handle as a \textit{standard stabilization} of $\Sigma$. 

\begin{lemma}\label{l:mainLemma}
Let $\Sigma$ be a closed embedded surface in a $4$-manifold $X$, and $\iota:\nu \Sigma \to  X$ denote the inclusion of its tubular neighborhood.  If $\iota_*:\pi_1(\partial (\nu \Sigma))\rightarrow \pi_1(X \setminus \Sigma)$ is surjective, then all handles attached to $\Sigma$ are trivial handles. 
\end{lemma}

This lemma is a straightforward generalization of the work of Boyle in \cite{B}, who classified the handle attachments to oriented surfaces in $S^4$, and its non-orientable counterpart, which was dealt with by Kamada in \cite{K2}. Below, will point out a few of the main features of these papers, which apply to any $\Sigma$ in $X$, and refer to these papers for the details. 

\begin{proof}
Suppose we attach a handle $h$ with attaching sphere at the points $a,b \in \Sigma$, where $\Sigma$ is oriented and $h$ respects the orientation. As shown by Boyle \cite[Section 2]{B}, up to isotopy, the resulting surface only depends on the homotopy class of the core of the handle, i.e. an arc from $a$ to $b$, called a \textit{cord} in this context. The set of cords from $a$ to $b$ can be acted upon in the following obvious ways: (1) we can add a meridian of $\Sigma$ to the homotopy class of the cord, or (2) we can pre-compose (resp. post-compose) with a loop in $\pi_1(\Sigma,a)$ (resp. $\pi_1(\Sigma,b)$) ---or rather, compose with a push-off of such a loop into $X \setminus \Sigma$. (See Figure~\ref{theslide}.) Boyle proves that any two handle additions give smoothly isotopic surfaces in $S^4$ if only if they arise from cords related by a sequence of these two operations \cite{B}, which, mutatis mutandis, works for surfaces in general 4-manifold $X$. Now if $\iota_*: \pi_1(\partial (\nu \Sigma)) \rightarrow \pi_1(X \setminus \Sigma)$ is onto, then any cord can be isotoped to lie in $\partial (\nu \Sigma) \subset \nu \Sigma$. Moreover, all cords represent a trivial handle. This is because $\pi_1(\partial (\nu \Sigma))$ is generated by $\pi_1(\Sigma)$ and the meridian to $\Sigma$ in $X$, and thus any cord can be related to a trivial one by a sequence of the above moves.

\begin{figure}[htbp]
\labellist
\small \hair 0pt
\pinlabel a at 123 23
\pinlabel b at 149 22
\endlabellist
\includegraphics[width=32mm]{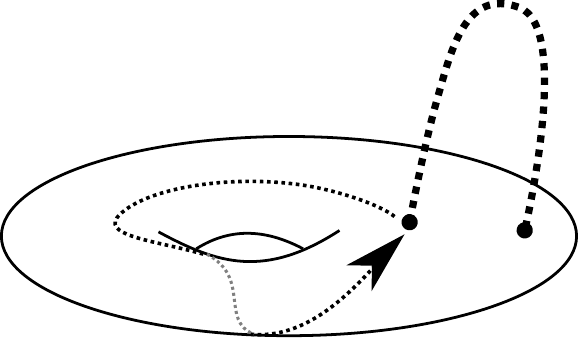}
\caption{The action of $\pi_1(\Sigma,a)$ on handles attached to $\Sigma$ at points $a$ and $b$.}
\label{theslide}
\end{figure}

The cases where either $\Sigma$ is non-orientable or the handle does not respect the orientation are quite similar. In those cases, it is still true that a handle is specified by a cord from $a$ to $b$, but we must also specify the framing along the cord \cite{K2}. Nevertheless, just as in the previous case, the surjectivity of $\pi_1(\partial (\nu \Sigma))\rightarrow \pi_1(X \setminus \Sigma)$ implies that the cord is trivial, and therefore the corresponding handle is trivial (which might or might not respect the local orientations).
\end{proof}

As the following elementary lemma demonstrates, for non-orientable surfaces, the type of the handle attachment (untwisted versus twisted) is often irrelevant.

\begin{lemma}\label{l:mainLemma2}
Let $\Sigma$ be non-orientable and $\pi_1(X  \setminus \Sigma)$ be a cyclic group generated by a meridian to $\Sigma$ in $X$. Then any handle attachment to $\Sigma$ 
is equivalent to a trivial untwisted handle attachment. More generally, if $\Sigma$ contains an orientation-reversing loop $\gamma$ that has a null homotopic push-off into $X \setminus \Sigma$, then any twisted handle attached to $\Sigma$ is equivalent to an untwisted handle with the same core (i.e. the framings along the same cord can be reversed). 
\end{lemma}

\begin{proof}
Take one of the two ends of the cord, say the point $a$,  and slide it around the loop $\gamma$ in $\Sigma$ back to its original location. By doing this, the framing of the handle is reversed, and the hypotheses guarantee that this does not change the homotopy class of the cord relative to its boundary. One can thus turn an untwisted handle attachment to a twisted one and vice versa.
\end{proof}


\vspace{0.1in}
\subsection{Stable equivalence} \

Here we will prove several stabilization results collected in Theorem~\ref{mainthm1}. We will handle the orientable and non-orientable surfaces separately.

\begin{theorem}\label{t:stabilizeT2}
Let $\Sigma_i$,  $i=0,1$, be a pair of closed orientable surfaces embedded in the interior of a compact oriented $4$-manifold $X$. If they are in the same homology class, then the $\Sigma_i$ become smoothly isotopic after a sufficient number of untwisted stabilizations. If, moreover, for $i=0,1$ the inclusion induced homomorphism $\pi_1(\partial (\nu \Sigma_i))\rightarrow \pi_1(X \setminus \Sigma_i)$ is surjective, then all stabilizing handles can be taken to be trivial. 
\end{theorem}



\begin{proof}
Let us first consider the following simplified scenario: $\Sigma_0$ and $\Sigma_1$ are disjoint, and the oriented link $\Sigma_0 \sqcup  - \Sigma_1$ bounds a \textit{Seifert 3-manifold}\footnote{Typically a Seifert manifold is defined to be orientable, but in this paper we will also allow allow non-orientable Seifert manifolds bounded by non-orientable surfaces.} $Y$ in $X$, that is, in this case a compact oriented $3$-manifold embedded into $X$ whose oriented boundary is $\Sigma_0 \sqcup  - \Sigma_1$. A relative Morse function on $(Y,\Sigma_0)$ induces a handle decomposition on $Y$. By standard arguments, we can assume that this is a self-indexing Morse function with only index $1$- and $2$- critical points. The induced handle decomposition implies that  $\Sigma_0$ plus the $1$-handles of $Y$ is isotopic to $\Sigma_1$ plus the $1$-handles of the dual handle decomposition of $Y$ (i.e. the original $2$-handles of $Y$, turned upside down).\footnote{This is the essence of the proof given in \cite{HK} for surfaces in $S^4$.} Because $Y$ is oriented, these handles are all untwisted handles.

This proof runs into a problem when $\Sigma_0$ and $\Sigma_1$ intersect, which is unavoidable when they represent a homology class with non-zero self-intersection. We will handle this more general situation in a way that will also cover the special case above. 

Begin by isotoping $\Sigma_0$ and $\Sigma_1$ such that they intersect transversely and non-trivially if they do not already intersect. The arguments to follow will essentially show that these surfaces are stably isotopic away from their intersection points, because they bound a \textit{relative Seifert manifold}, a certain $3$-manifold with corners that we will discuss shortly. Let $\Sigma_0 \cap \Sigma_1 = \{p_1,\ldots p_n\}$. Orient $\Sigma_0$ and $\Sigma_1$ such that their union $\Sigma_0 \cup -\Sigma_1$ is null-homologous in $X$. For each $j=1, \ldots, n$, let $D_j$ be a $4$-ball centered at $p_j$, disjoint from each other, and small enough that $\partial D_j \cap (\Sigma_0 \cup \Sigma_1)$ is a Hopf link. Denote the annular Seifert surface spanning these Hopf links $A_j$. Let $\hat{X}$ be the complement of the interiors of all $D_j$, $j=1, \ldots, n$, in $X$ and set $\hat{\Sigma}_i = \hat{X}\cap \Sigma_i$, for each $i=0,1$.

\vspace{0.2cm} 
\noindent \textit{Claim: The union of $\hat{\Sigma}_0, -\hat{\Sigma}_1$, and the $A_i$ bounds an oriented $3$-manifold $Y$ with corners with the following properties: $Y$ is a cobordism between manifolds with boundary, a cobordism from $(\hat{\Sigma}_0,\partial \hat{\Sigma}_0)$ to $(\hat{\Sigma}_1,\partial \hat{\Sigma}_1)$ which is a product cobordism on the boundary (i.e. $A_j$ on the boundary).} \\

Assuming this fact for the moment, we can induce a relative handle structure on the manifold with corners $Y$, which, just as in the simpler case discussed in the beginning, shows that $\hat{\Sigma}_0$ and $\hat{\Sigma}_1$ become isotopic after adding to each the untwisted $1$- and $2$-handles of $Y$, respectively. Evidently, this isotopy follows the path of the annulus in $\partial \hat{X}$. Such an isotopy clearly extends to the $D_j$, i.e. the isotopy of the stabilized $\hat{\Sigma}_0$ to the stabilized $\hat{\Sigma}_1$ extends to an isotopy of the stabilized $\Sigma_0$ to the stabilized $\Sigma_1$. Lastly, if $\pi_1(\partial \nu \Sigma_i)\rightarrow \pi_1(X \setminus \Sigma_i)$ is surjective, then all of the handles are trivial by Lemma \ref{l:mainLemma}.

It remains to prove the above claim, which will complete the proof of our theorem. Our proof of this claim will follow the familiar scheme to produce Seifert surfaces: Letting $M = \hat{X} \setminus \nu(\hat{\Sigma}_0 \cup -\hat{\Sigma}_1)$, we will construct a map $M \to S^1$ with $Y$ as the preimage of a point. Since such a proof seems to be absent from the literature in the case of surfaces with boundary, we will be explicit about a few of the key details.

First note that there is a cohomology class $\alpha \in H^1(M)$ such that $\alpha[x] = 1$, where $x$ is any oriented meridian of either $\hat{\Sigma}_0$ or $-\hat{\Sigma}_1$. To show this, it is sufficient to find a properly embedded $3$-manifold (with boundary) in $M$ that intersects any meridian once; $\alpha$ will be its Poincar\'{e} dual. Such a $3$-manifold exists for the following reason: In $X$, smooth $\Sigma_0 \cup -\Sigma_1$ at its double points. This will be a connected null-homologous surface, which we denote by $\Sigma$. So $\Sigma$ bounds an oriented $3$-manifold in $X$ (see e.g. \cite{Kirby}). But the intersection of this $3$-manifold with $M$ is exactly the $3$-manifold with boundary (and corners) that we desire.

Now, represent $\alpha$ as a map $f:M \rightarrow S^1$, which we will homotope so that $f^{-1}(pt)$ is the desired (relative) Seifert $3$-manifold. We will first homotope $f$ along the boundary of $M$, and then extend that homotopy to be smooth in the interior as well. To simplify our exposition, assume that $\Sigma_0$ and $\Sigma_1$ intersect at one point $p_1$. The boundary of $M$ is $\hat{\Sigma}_0 \times S^1 \cup \hat{\Sigma}_1 \times S^1 \cup H$, where $H$ is the complement of the Hopf-link in $S^3$, the boundary of the $4$-ball $D_1$ we took around $p_1$. In the $\hat{\Sigma}_0 \times S^1$ piece, $f$ is dual to the meridian $pt \times S^1$, so it is homotopic to the projection onto the second factor, $\hat{\Sigma}_0 \times S^1 \rightarrow S^1$, perhaps after first composing with a self-homeomorphism of $\hat{\Sigma}_0\times S^1$. Similarly, $f$ is dual to the meridian of the $\hat{\Sigma}_1 \times S^1$ piece, so it is homotopic to a projection there as well. 

Since Hopf link is a fibered link with annuli fibers, there exists a map $g: H \rightarrow S^1$ such that $g^{-1}(pt)$ is the annular Seifert surface of the Hopf link. Next, we would like show that $f$ and $g$ are homotopic on $H$. Let $\beta$ be the class in $H^1(H)$ corresponding to $g$. It suffices to show that $\alpha|_H = \beta$ in $H^1(H)$. Here $H_1(H)$ is generated by the meridians $\mu, \mu'$ to each component of the Hopf link.  We then have 
\[ \langle \alpha|_H,  [\mu] \rangle = \langle\alpha|_H, [\mu']\rangle = 1 = \langle \beta, [\mu] \rangle = \langle \beta, [\mu']\rangle \, , \] 
because $\alpha$ is dual to $[\mu]=[\mu']$ in $H^1(M)$, and because $\langle \beta, [\mu] \rangle = g^{-1}(pt) \cdot \mu$, where $\mu$ intersects the Seifert surface $g^{-1}(pt)$ once (and the same goes for $\mu'$).  So $\alpha|_H = \beta$, implying that $f$ is homotopic to $g$ on $H$. We can of course homotope $f$ in the same manner around any number of points in $\Sigma_0 \cap \Sigma_1$, thus the argument extends to any number of intersection points $\{p_0, \ldots, p_n\}$.

Finally, with this choice of map $f$ after a homotopy, we have that $Y=f^{-1}(pt)$ is the oriented Seifert $3$-manifold with corners that we sought (or rather, it extends from $M$ to a Seifert manifold in $\hat{X}$).
\end{proof}

\smallskip
Next, we will prove a version of the above theorem for non-orientable surfaces. The crucial difference in this case is that the normal bundle of such a surface is no longer determined by its homology class (which is trivial as an integral class) but rather by its \textit{normal Euler number}. (See e.g. \cite{Mas} for the normal bundles of non-orientable surfaces in $S^4$.)

\begin{theorem}\label{ThmStabNonorient}
Let $\Sigma_i$,  $i=0,1$, be a pair of closed non-orientable surfaces embedded in the interior of a compact oriented $4$-manifold $X$. If they are in the same homology class, and have the same normal Euler number, then the $\Sigma_i$ become smoothly isotopic after a sufficient number of stabilizations by handle attachments. If, moreover, the inclusion induced homomorphism $\pi_1(\partial (\nu \Sigma_i))\rightarrow \pi_1(X \setminus \Sigma_i)$ is surjective for $i=0,1$, then all handles can be taken to be standard. If in addition $\Sigma_i$ contains an orientation-reversing loop $\gamma$ that has a null homotopic push-off into $X \setminus \Sigma_i$, then these trivial handles can be assumed to be untwisted.
\end{theorem}

\begin{proof}

The proof is similar to the orientable case. That is, we'll first isotope the surfaces so that they intersect non-trivially, and then --using the notation of the previous theorem-- find a Seifert manifold with corners $Y$ in $X$ to construct the isotopy. The only extra care we need to take is in finding the Seifert manifold.

To do this, we will first show that there is a cohomology class in $H^1(M, \mathbb{Z}_2)$ characterized by evaluating to 1 on any meridian of $\hat{\Sigma}_0 \cup \hat{\Sigma}_1$. As before, this is equivalent to finding a Seifert manifold for the connected surface $\Sigma$ in $X$. Since such a Seifert manifold intersects any meridian once, the  Poincare dual of such a Seifert manifold pulled back to $M$ will be the $\alpha$ we desire. But why does $\Sigma$ in $X$ have such a Seifert manifold? 

First of all, a meridian to $\Sigma$ must be homologically essential in $H_1(X\setminus \nu\Sigma, \mathbb{Z}_2)$, since otherwise there would be a surface intersecting $\Sigma$ once, which is impossible given $\Sigma$ is homologically trivial. The dual  $\alpha \in H^1(X\setminus \nu\Sigma, \mathbb{Z}_2)$ to this meridian can be represented by a map $f:X\setminus\nu\Sigma \rightarrow RP^{\infty}$. 

\vspace{0.2cm} 
\noindent \textit{Claim: We can homotope $f$ so that on the boundary $\partial\nu\Sigma$, which is an $S^1$ bundle over $\Sigma$, the inverse image of a codimension-$1$ copy of $RP^{\infty}$ is a section of the $S^1$ bundle.} \\

Assuming the claim, we can extend this homotopy to all of $X\setminus\nu\Sigma$, smooth it in the interior, and then the preimage of $RP^{\infty}$ in $X\setminus \nu\Sigma$ is a Seifert manifold to $\Sigma$. Since $\alpha$ pulled back to $M$ evaluates to $1$ on every meridian of either $\hat{\Sigma}_0$ or $\hat{\Sigma}_1$, we can similarly homotope the corresponding map $f:M \rightarrow RP^{\infty}$ such that the preimage of a codimension-1 copy of $RP^{\infty}$ is a section along $\partial \nu (\hat{\Sigma}_0 \cup \hat{\Sigma}_1)$ and the Hopf link on $H$. Using this method we have found a Seifert manifold $Y$ with corners as before, and the rest of the argument follows exactly as in the orientable case. Note however, that in this case our Seifert manifold won't be orientable, so the handles that we add are not automatically untwisted. On the other hand, when $\pi_1(\partial (\nu \Sigma_i))\rightarrow \pi_1(X \setminus \Sigma_i)$ is surjective, and $\Sigma_i$ contains an orientation reversing loop that pushes off to a null-homotopic loop in $X\setminus \Sigma_i$, then all of the handles are equivalent to untwisted handles by Lemma \ref{l:mainLemma2}.

It only remains to prove the claim.

Let $N = \partial (\nu\Sigma)$ and recall that this is an $S^1$ bundle over $\Sigma$ with trivial Euler number. By definition, $\alpha$ restricted to $N$ is dual to any $S^1$ fiber, a meridian of $\Sigma$. If we can find  find a map $f':N \rightarrow RP^{\infty}$ such that the preimage of a codimension-1 copy of $RP^{\infty}$ is a section of the $S^1$ bundle, then it must be homotopic to $f$, because it represents a class in $H_2(N,\mathbb{Z}_2)$ that is similarly dual. To construct such an $f'$, let for an interval $I$, $I \tilde{\times} \Sigma$ denote a neighborhood of a section of the $S^1$-bundle. This $I$-bundle is classified by a map $0\times \Sigma \rightarrow RP^2$, which extends to a bundle map from  $I \tilde{\times} \Sigma$ to the tautological bundle over $RP^2$, and by collapsing we get a map from all of N to the Thom space of the tautological bundle (i.e. $RP^3$), such that the inverse image of $RP^2$ is $0\times \Sigma$ (i.e. a section of the $S^1$ bundle $N$). This is easily turned into a statement about $RP^n$ for arbitrarily large $n$.
\end{proof}

\begin{remark}
The extra requirement about the Euler class is necessary, because a non-orientable surface can be null-homologous, but if its normal bundle does not have a non-vanishing section, it will fail to have a Seifert manifold. Seifert manifolds of non-orientable surfaces in $S^4$ are investigated in \cite[Theorem 3.8]{K} and \cite[Section 6]{GL}.
\end{remark}

\smallskip

With the above theorems in hand, we can now derive the following corollary for exotically knotted surfaces:

\begin{cor}\label{CorStabExotic}
Let $\Sigma_i$, $i=0,1$, be a pair of exotically knotted surfaces in $X$, where $\pi_1(X\setminus \Sigma_i)$ is a cyclic group generated by a meridian of $\Sigma_i$ in $X$. Then they become smoothly isotopic after a same number of standard (untwisted) stabilizations. 
\end{cor}

\begin{proof}
An exotic pair of $\Sigma_0$ and $\Sigma_1$ are topologically isotopic, so they are homeomorphic and are in the same homology class. Furthermore, if they are non-orientable, then both will have the same normal Euler number. So by the above theorems, they become smoothly isotopic after adding enough number of handles, necessarily the same number to each. On the other hand, the assumption on $\pi_1(X \setminus \Sigma_i)$ implies both that the inclusion induced homomorphism $\pi_1(\partial (\nu \Sigma_i))\rightarrow \pi_1(X \setminus \Sigma_i)$ is surjective for $i=0,1$, so the handles can be taken to be standard, moreover can be taken to be untwisted.
\end{proof}

\vspace{0.2in}

\section{How many stabilizations are needed?}

In this section we will develop the tools necessary to explicitly stabilize surfaces in concrete situations. Insofar as many exotic surfaces are constructed using techniques that can be viewed as $3$-dimensional modifications crossed with $S^1$, we will begin with some notation and a few basic facts for handles in that context. After that we will move on to analyzing all of the different techniques currently known for constructing exotic knottings of surfaces, and show for each and every one that the resulting knotted surfaces become smoothly isotopic after a single untwisted standard stabilization in the sense of previous section.

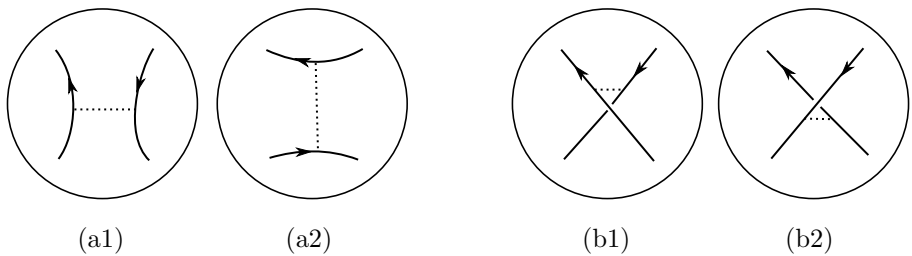
\begin{figure}[htbp]
\labellist\hair 0pt
\vspace{0.1in}
\pinlabel (a1) at 38 -9
\pinlabel (a2) at 122 -9
\pinlabel (b1) at 242 -9
\pinlabel (b2) at 324 -9
\vspace{0.1in}
\endlabellist
\vspace{0.1in}
\includegraphics[width=120mm]{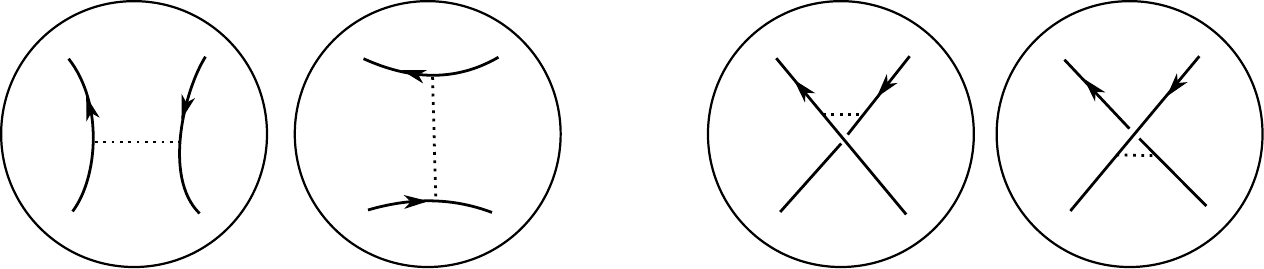}
\caption{Local stable equivalences.}
\label{f:crossings}
\end{figure}

Suppose we are given an $S^1 \times D^3$ in our $4$-manifold that intersects a surface in a pair of annuli. Here we will denote these annuli by arcs in $D^3$, where the $S^1$ factor will always be implied, as in Figure \ref{f:crossings}. A dotted line between the two components will specify a handle as follows: the dotted line represents a cord in $pt \times D^3$ which is the core of the handle. If the surface is orientable, we will assume that the handle is untwisted, i.e. it respects the orientation. In the case that the surface is non-orientable, a dotted line corresponds to attaching a handle oriented with respect to some local orientation.

We will use the following proposition repeatedly:

\begin{prop}\label{p:mainProp}
The surface given in Figure~\ref{f:crossings} (a1), a pair of annuli with a handle attached as shown, is smoothly isotopic to the surface given in (a2), relative to the boundary. Similarly, the surfaces given in (b1) and (b2) are isotopic. 
\end{prop}

It is moreover true that the surfaces in Figure \ref{f:crossings} (a1) and (a2) are isotopic to those in (b1) and (b2) if we take the handles in (a1) and (a2) to be attached with the opposite orientation as the handles in (b1) and (b2).

\begin{proof}
Proofs of these claims are given in Figures~\ref{f:theswitch} and \ref{f:crossings2}. 

\begin{figure}[htbp]
\includegraphics[height=100mm]{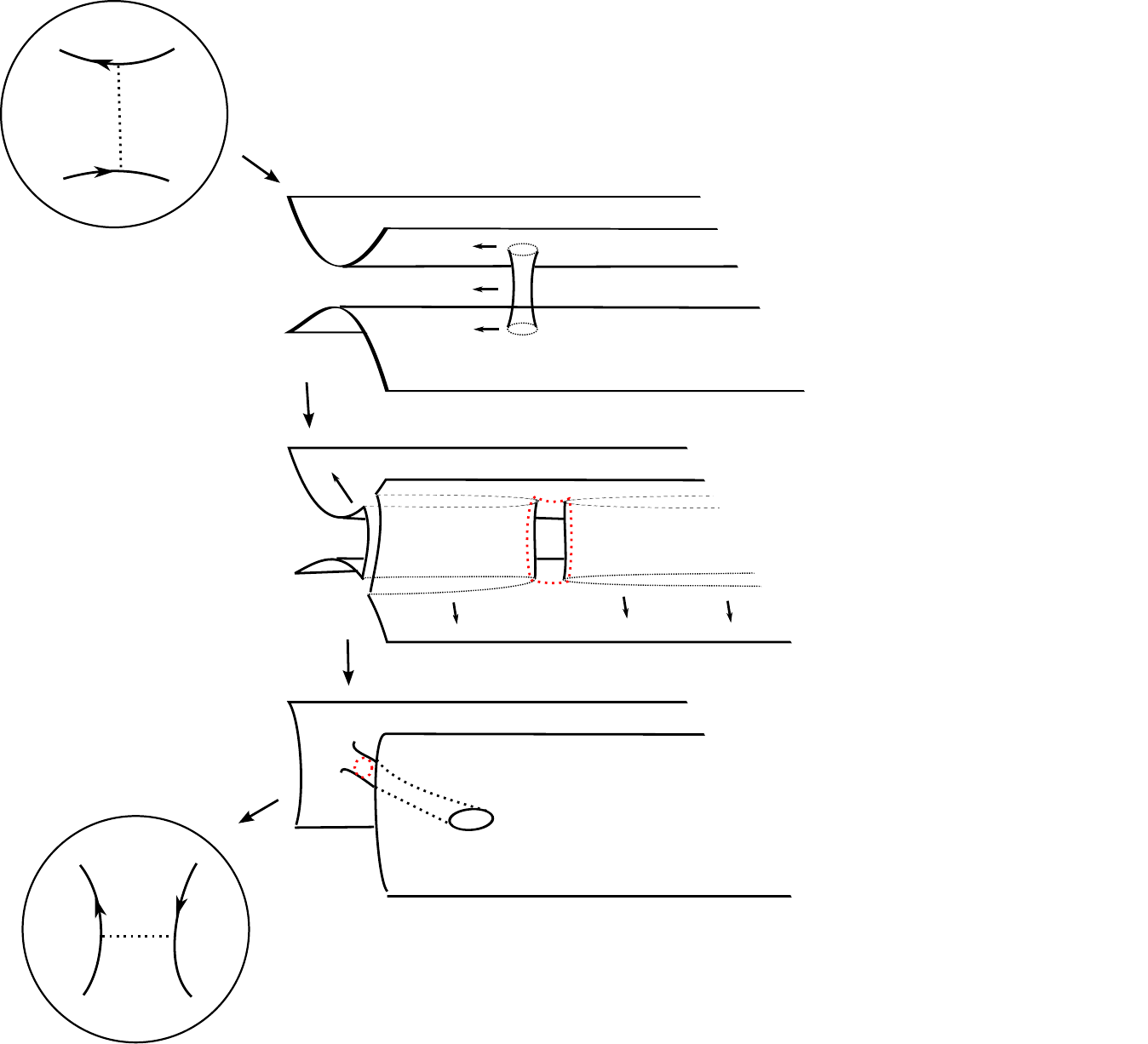}
\vspace{0.2in}
\caption{An explicit isotopy rel $\partial$ showing the equivalence of the horizontal annuli plus handle with the vertical annuli plus handle.}
\label{f:theswitch}
\end{figure}

To see the equivalence of (a1) and (a2), one can imagine making the handle between the two sheets seen in Figure~\ref{f:theswitch} larger and larger, pushing it all the way around in the $S^1$ direction

\begin{figure}[htbp]
\vspace{0.2in}
\includegraphics[height=60mm]{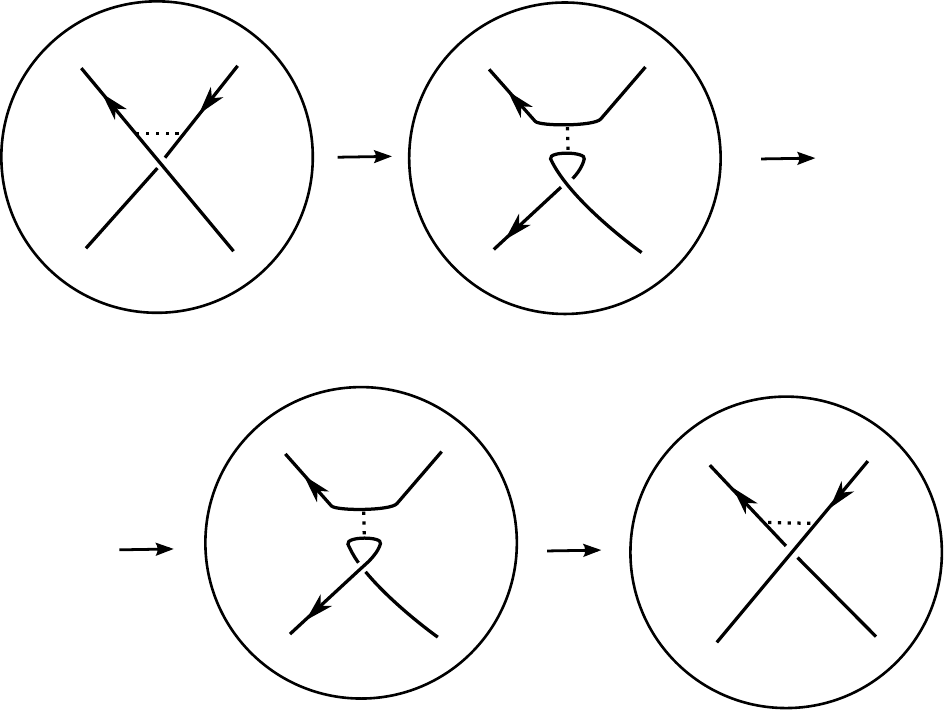}
\vspace{0.2in}
\caption{The equivalence of (b1) and (b2) by appealing locally to the the equivalence of (a1) and (a2) for the surfaces in Figure~\ref{f:crossings} }
\label{f:crossings2}
\end{figure}
\vspace{0.2in}

On the other hand the equivalence of the surfaces in Figure \ref{f:crossings} (b1) and \ref{f:crossings}(b2) can be seen by applying the (a1)-(a2) equivalence locally, as demonstrated in Figure~\ref{f:crossings2}.
\end{proof}

%

\subsection{Rim surgery} \
 
In their formative work on exotic surfaces, Fintushel and Stern introduced the rim surgery technique, which allows one to construct infinitely many exotic knottings of an orientable surface $F$ of positive genus and non-negative self intersection in a simply-connected $4$-manifold $X$, provided $(X, F)$ has non-trivial relative Seiberg-Witten invariants \cite{FintushelStern1}. (As Mark showed in \cite{Mark}, relative Heegaard-Floer invariants can be used to the same effect, and he uses this to extend the construction to surfaces with negative self intersection.) Variations on this technique have led to a myriad of examples of exotically knotted surfaces, as we will discuss shortly.

First, a quick review of the construction. Begin with an annular submanifold of an embedded surface. A neighborhood of this annulus in the ambient $4$-manifold can be viewed as $(S^1 \times D^3, S^1 \times D^1)$. Rim surgery is the process of replacing $S^1 \times D^1$ in this neighborhood with a knotted arc, $(S^1\times D^3, S^1 \x D^1 \# K)$. For example, replace the annulus represented by Figure \ref{f:rimsurgery}d with the annulus in \ref{f:rimsurgery}a. (Recall that we cross each picture with $S^1$). Using relative Seiberg-Witten invariants, Fintushel and Stern show that this construction often\footnote{It is only known that $\Sigma_{K_1}$ and $\Sigma_{K_2}$ are smoothly distinct in the case that $\Sigma$ is symplectically embedded, the annulus' core is nontrivial in $H_1(\Sigma)$, and $K_1$ and $K_2$ have different Alexander polynomials.} produces smoothly exotic embeddings.
On the other hand, by restricting themselves to surfaces whose complements have trivial fundamental group \cite{FintushelStern1}, the rim-surgered surfaces are indeed all topologically isotopic to the original surface. 

\begin{figure}[! htbp]
\labellist\hair 0pt
\pinlabel (a) at 88 117
\pinlabel (b) at 197 117
\pinlabel (c) at 144 7
\pinlabel (d) at 266 7
\pinlabel ... at 252 159
\endlabellist
\includegraphics[height=60mm]{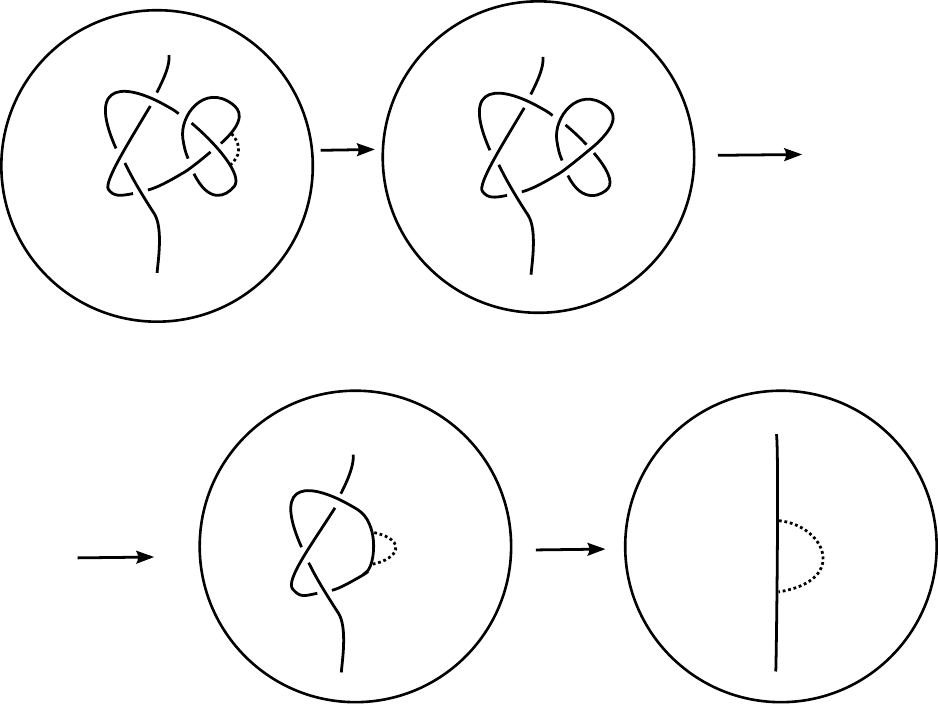}
\caption{Stabilizing rim-surgery.}
\label{f:rimsurgery}
\end{figure}
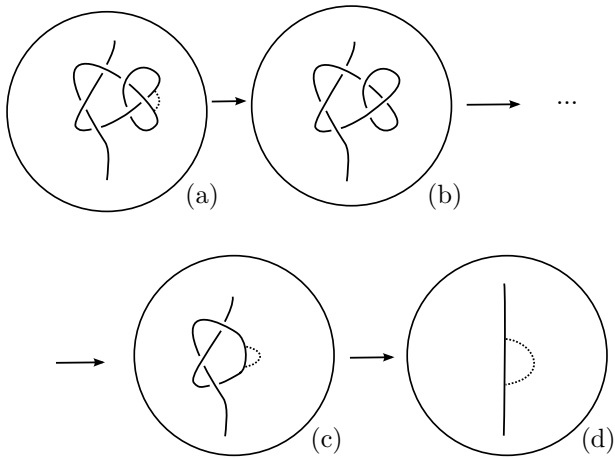

Here is how rim-surgered surfaces can be seen to become all smoothly isotopic after a single stabilization:

At a crossing of the knotted arc, add a handle as in Figure \ref{f:rimsurgery}a. By Lemma \ref{l:mainLemma} this is a trivial handle. Using Proposition \ref{p:mainProp}, we can switch the crossing. The resulting surface without the added handle must still have simply-connected complement. This is because it can be thought as a surface obtained by a similar rim-surgery along the original surface with a different knot, and all rim surgeries on a surface with simply connected complement result in a surface with simply connected complement. So, the resulting handle represented by the dotted line in Figure \ref{f:rimsurgery}b is trivial by our observations in the previous section.

This trivial handle can be re-used, sequentially applying Lemma \ref{l:mainLemma} and Proposition \ref{p:mainProp} until all that is left is an unknotted arc and a trivial handle.  In this way we get back the non-rim-surgered surface after adding a handle.

\vspace{0.1in}
\subsection{Twist rim surgery} \

\textit{Twist rim surgery} is a variation of the rim surgery technique introduced by Kim in \cite{Kim} in order to produce exotic knottings of surfaces with finite cyclic fundamental group complements (and some other finite groups in controlled settings), as further explored in the works of Kim and Ruberman \cite{KimRuberman1, KimRuberman2, KimRuberman3}.

Informally, n-twisted rim surgery amounts to performing rim surgery while the knotted arc, $D^1 \#K$, spins n-times as it goes around in the $S^1$ direction. Locally we call the complement of the surgered annulus $S^1 \tilde{\times} C(K) \subset S^1\times D^3$, and call the surgered surface itself $\Sigma_K(n) \subset X$. In the case that $\pi_1(X \setminus \Sigma) = \mathbb{Z}_d$ and $(n,d) = 1$, they prove that $n$-twist rim surgery does not change the fundamental group (i.e. $\pi_1(X \setminus \Sigma) = \pi_1(X \setminus \Sigma_K(n))$, see \cite{KimRuberman1}) and that the n-twist rim surgered surface is topologically isotopic to the original one. To show that $\Sigma$ becomes smoothly equivalent to $\Sigma_K(n)$ after a single standard stabilization, one simply proceeds as in the rim-surgery case above, noting that Proposition \ref{p:mainProp} can be applied to the spinning crossing. In this case, the added handle must be trivial by Lemma \ref{l:mainLemma}.

In the case that $\pi_1(X \setminus \Sigma)$ is non-cyclic, we must be more careful. In \cite{KimRuberman2}, they show that, they show that 1-twist rim surgery does not change the fundamental group of the surface complement (i.e.  $\pi_1(X \setminus \Sigma) = \pi_1(X \setminus \Sigma_K(1)$) and they construct concrete instances where 1-twist rim surgery yields exotic knottings. As before, we can unknot the 1-twist rim surgery crossing by crossing, but we need to check that the handle at each step (e.g. the handles from Figure \ref{f:rimsurgery}) is trivial, so that we can reposition them for subsequent steps. According to \cite[Proposition 2.3]{KimRuberman1}, the image of  $\pi_1(S^1 \tilde{\times} C(K))$ in $\pi_1(X \setminus \Sigma_K(1))$ is the cyclic subgroup generated by the meridian of $\Sigma_K(1)$. Therefore, every handle, since it is attached within $S^1\times D^3$, is homotopic to a handle attached along some number of meridians to the surface in the X, and therefore, as in the proof of Lemma \ref{l:mainLemma} such handles must in fact be trivial.  

\vspace{0.1in}
\subsection{Finashin's annulus rim surgery.} \

In \cite{Finashin1, Finashin2}, Finashin constructs exotic knottings\footnote{It is shown in \cite{KimRuberman1} that most of these smoothly knotted surfaces are topologically isotopic, thus exotically knotted.} of algebraic curves by performing the local modification of replacing the double annulus in Figure \ref{f:finalg}e with Figure \ref{f:finalg}a, which can be thought of as performing rim surgery along these two annuli simultaneously. To unknot this after a single stabilization, one adds handle as in Figure \ref{f:finalg}a, which is trivial by Lemma \ref{l:mainLemma}. Apply Proposition \ref{p:mainProp} (Figure \ref{f:finalg}b), followed by an isotopy of the annuli (Figure \ref{f:finalg}c). The additional handle, dragged along during this isotopy becomes knotted in the $3$-ball  (Figure \ref{f:finalg}c), but it can be unknotted in the $4$-manifold  (Figure \ref{f:finalg}c--d). 

\begin{figure}[htbp]
\labellist\hair 0pt
\pinlabel (a) at 87 117
\pinlabel (b) at 196 117
\pinlabel (c) at 310 117
\pinlabel (d) at 223 7
\pinlabel (e) at 346 7
\endlabellist
\includegraphics[height=60mm]{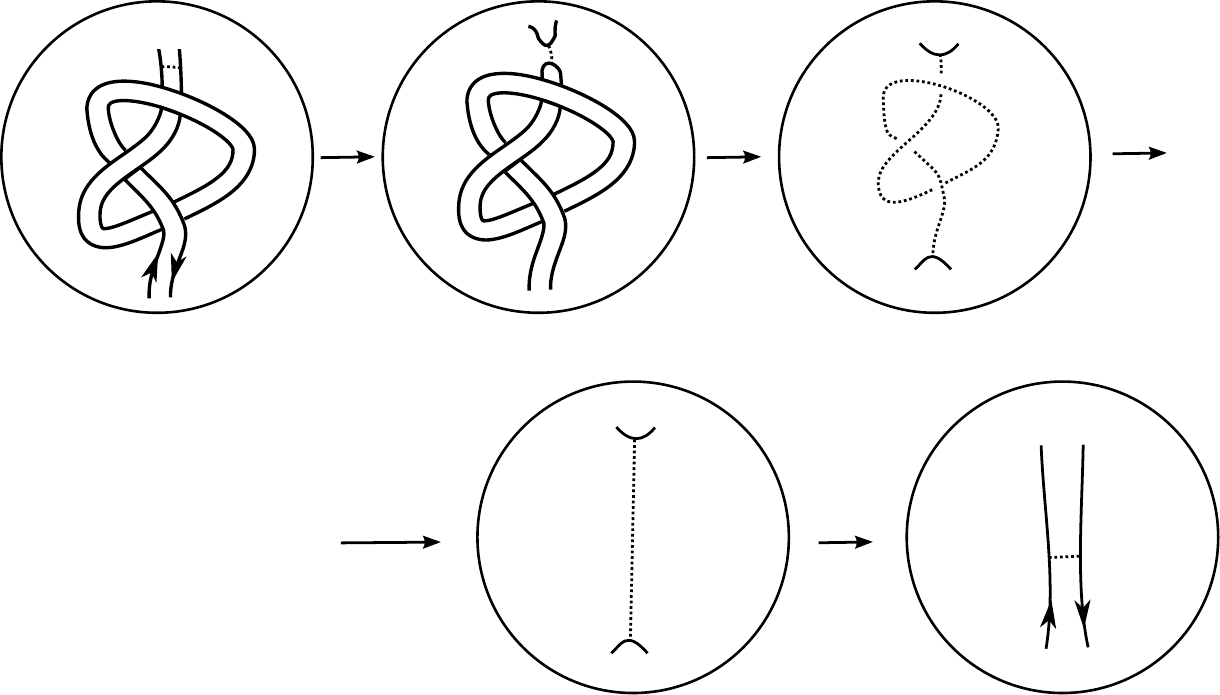}
\vspace{0.1in}
\caption{Stabilizing Finashin's algebraic curves.}\label{f:finalg}
\end{figure}
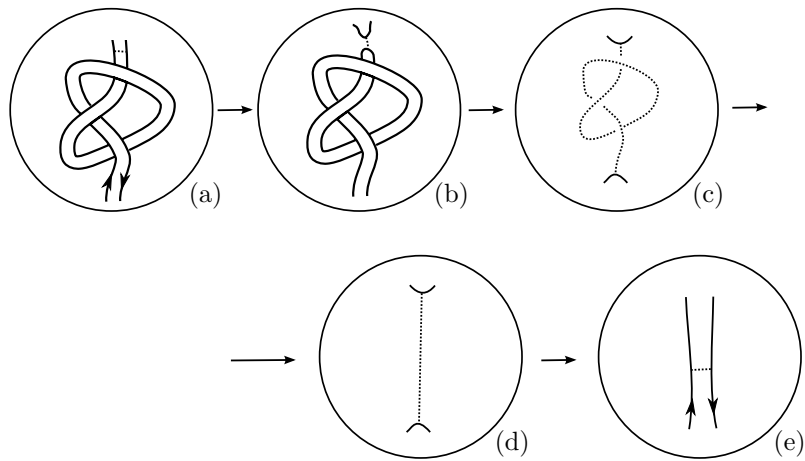

Finally,  a second application of Proposition \ref{p:mainProp} (Figure \ref{f:finalg}d), gives back the original surface with a trivial handle attached. So the knotted surface and the algebraic curve become smoothly isotopic after a single standard stabilization.

\vspace{0.1in}
\subsection{Finashin-Kreck-Viro tangle surgery construction} \

In their pioneering work,  Finashin, Kreck and Viro \cite{FKV} produced the first examples of infinitely many exotic knottings of surfaces in $4$-manifolds.  By performing equivariant logarithmic transforms on the elliptic surface $E(1)$, they obtained involutions on $E(1)$, whose fixed point loci descend to the quotient as exotic knottings of the standard\footnote{Viewing the standard embedding of $\RP$ in $S^4$ with Euler number $-2$ (resp. $+2$) as the image of the fixed point set under the complex involution on $\CP$ (resp. $\CPb$), the standard embedding of $\#10 \RP$ in $S^4$ with normal Euler number $16$ is defined as the quotient of the induced involution on the connected sum $E(1)= \CP \# 9 \CPb$. The surgery in $S^4$ we will discuss here descends from the equivariant surgery on $E(1)$.}  $\#10 \RP$ in $S^4$ with normal Euler number $16$.

These knotted surfaces can be constructed locally by the following \textit{tangle surgery} on the standard surface in $S^4$: Locally replace Figure \ref{f:fvkundo}f with \ref{f:fvkundo}a, that is, replace the two annuli by the two knotted annuli, ignoring the stabilizing handles prescribed by the dotted lines for the moment. Here the number of half-twists on the right and left parts of Figure \ref{f:fvkundo}a are assumed to be relatively prime.\footnote{The authors do this only for $2$ half-twists on the left. In this case the branched cover of the surface is $E(1)_{2,q}$. Using more general $p,q$ extends the equivariant family of exotic double covers to $E(1)_{p,q}$.} 

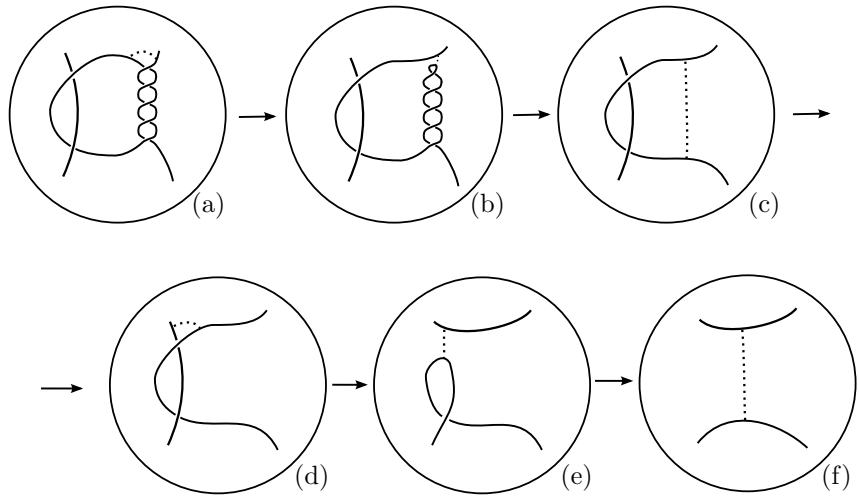
\begin{figure}[htbp]
\labellist\hair 0pt
\pinlabel (a) at 83 120
\pinlabel (b) at 199 120
\pinlabel (c) at 314 120
\pinlabel (d) at 126 7
\pinlabel (e) at 236 7
\pinlabel (f) at 344 7
\endlabellist
\vspace{0.1in}
\includegraphics[height=65mm]{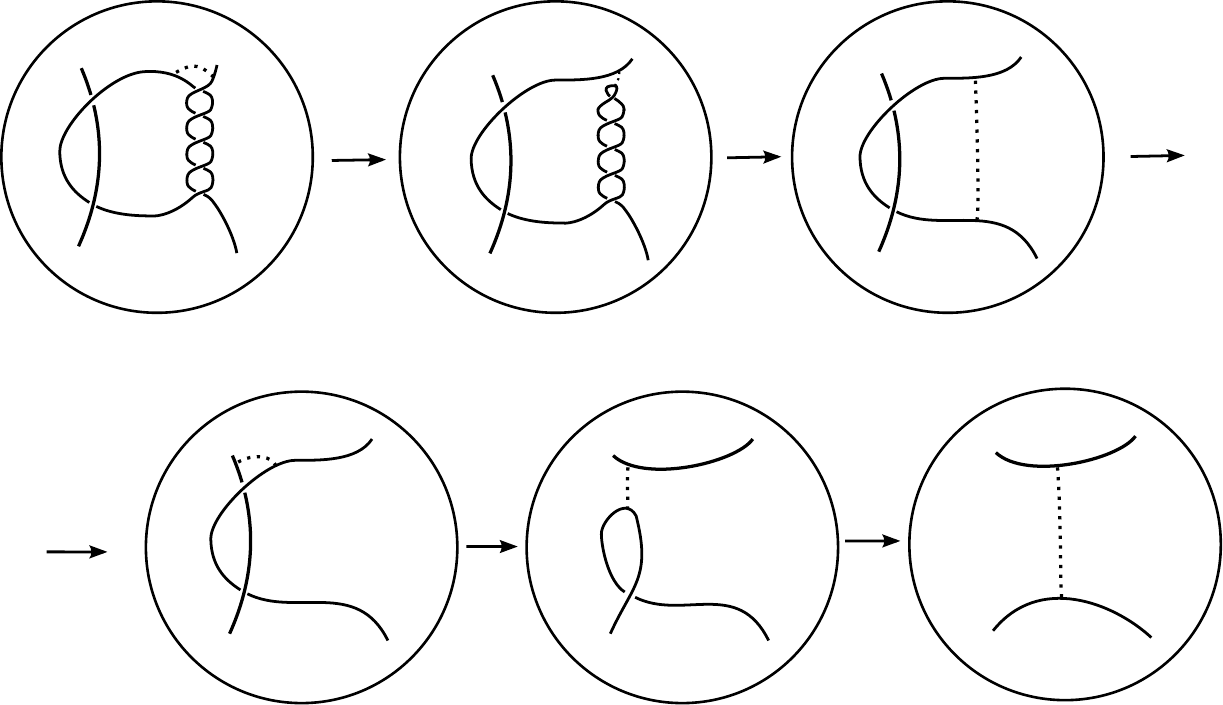}
\vspace{0.1in}
\caption{Stabilizing non-orientable surfaces of Finashin-Kreck-Viro.}\label{f:fvkundo}
\end{figure}

To see that these become equivalent after one standard stabilization, we follow the sequence given in Figure \ref{f:fvkundo}. Arguing that these stabilizations can be done by trivial handle attachments requires a closer look at the fundamental group of the surface complements in these modifications.

\vspace{0.2in}
\begin{figure}[htbp]
\labellist \small \hair 0pt
\pinlabel (a) at 100 0
\pinlabel $\tau$ at 155 50
\pinlabel (b) [Bl] at 242 0
\pinlabel u [Bl] at 276 76
\pinlabel v [Bl] at 296 76
\pinlabel x [Bl] at 273 11
\pinlabel y [Bl] at 295 11
\endlabellist
\includegraphics[height=20mm]{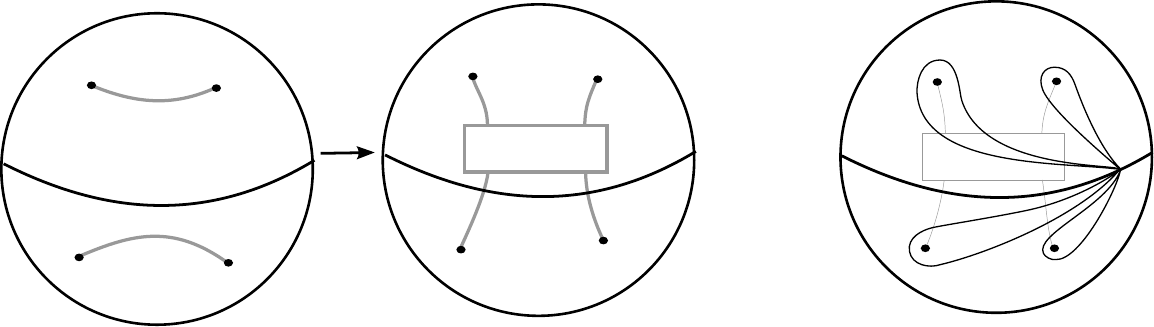}
\caption{A general tangle surgery (a), will not change the fundamental group when $\pi_1(D^3\setminus \tau)$ is abelian after adding the prescribed relations.}
\label{f:fvkundo2}
\end{figure}
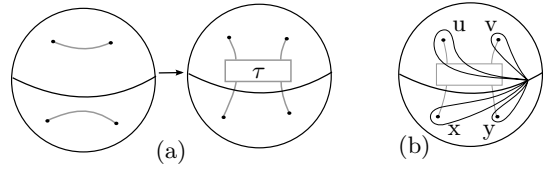

It is shown in \cite{FKV} that the surface in $S^4$ corresponding to Figure \ref{f:fvkundo}a has an abelian complement, so the handle in Figure \ref{f:fvkundo}a is a trivial handle. Assume it is oriented according to some local orientation (a local orientation that, in this case, we cannot extend even to the surface depicted). Each subsequent step is either a consequence of Proposition \ref{p:mainProp} or Lemma \ref{l:mainLemma2}. The reason we can use Lemma \ref{l:mainLemma2} to reposition the handle at will, and not need to worry about the local orientation, is as follows: According to \cite[Section 3.3]{FKV}, if the complement of the original surface has cyclic fundamental group, the fundamental group of the complement will still be cyclic after the arbitrary tangle surgery depicted in Figure \ref{f:fvkundo2}a as long as $\pi_1(D^3\setminus \tau)$ becomes abelian after adding the relations $x = y^{\alpha}= u^{\beta} = v^{\gamma}$ with $\alpha, \beta,\gamma = \pm 1$, where the curves $x,y,u$ and $v$ are in $\partial D^3$ as in Figure \ref{f:fvkundo2}b. This is clearly true for $\tau$ as in Figures \ref{f:fvkundo}c-f. Therefore, by Lemma \ref{l:mainLemma2}, the handles attached at those stages are trivial handles.

\subsection{Equivariant knot surgery} \

In \cite{Finashin3}, Finashin gives a construction of exotic knottings of a ``smaller'' non-orientable surface, $\# 6 RP^2$  in $S^4$ with normal Euler number $8$. In the same spirit as \cite{FKV}, these examples arise as the fixed sets of involutions on exotic $\CP \# 5 \CPb$s. (These exotic manifolds come from \cite{PSS}). In his beautiful paper, Finashin describes equivariant versions of \textit{knot surgery} \cite{FSknotsurgery}, \textit{double node surgery} \cite{FSdoublenode} and the \textit{rational blow-down} \cite{FSrationalblowdown} constructions of Fintushel and Stern. 

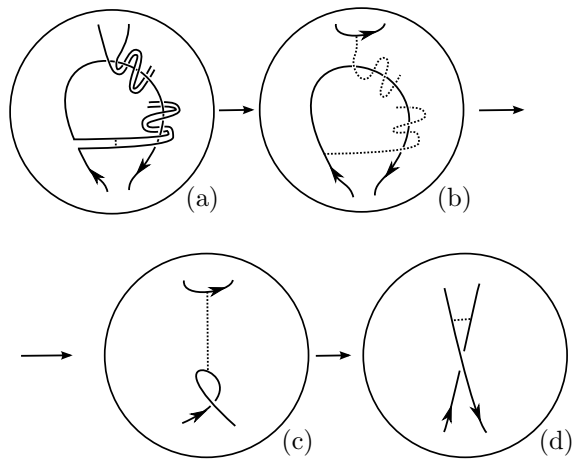
\begin{figure}[htbp]
\labellist\hair 0pt
\pinlabel (a) at 85 114
\pinlabel (b) at 196 114
\pinlabel (c) at 127 7
\pinlabel (d) at 241 7
\endlabellist
\vspace{0.2in}
\includegraphics[height=60mm]{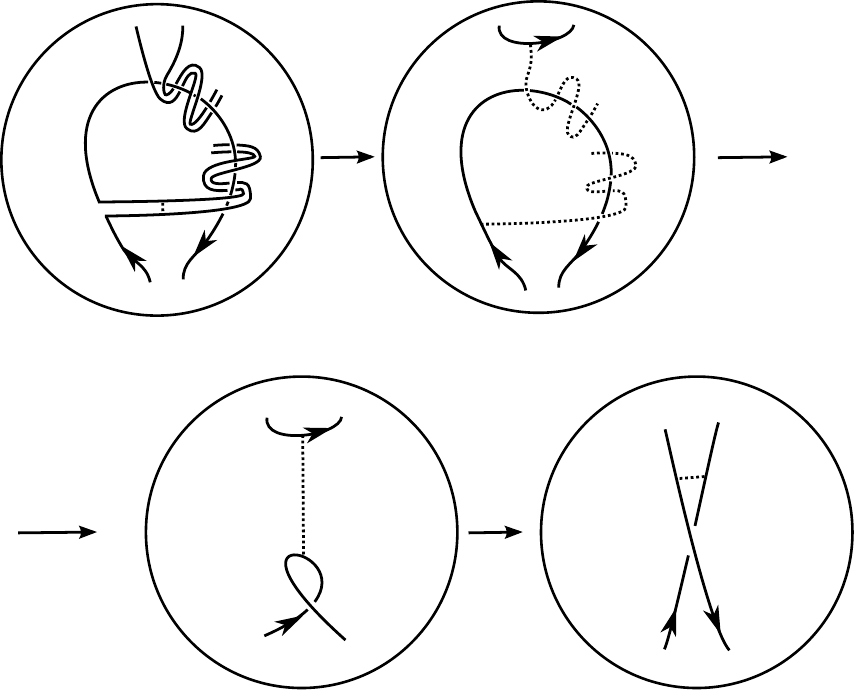}
\vspace{0.1in}
\caption{Stabilizing Finashin's equivariant knot surgery.} \label{f:finknot}
\end{figure}

Here the equivariant knot surgery -- performed in a double node neighborhood -- alters the surface in $S^4$ by exchanging Figure~\ref{f:finknot}d for Figure~\ref{f:finknot}a. (See \cite[Figure 6e]{Finashin3}). The stabilized surface is unraveled using Proposition \ref{p:mainProp} followed by an isotopy, followed by the Proposition~\ref{p:mainProp}  again. The handles are all trivial for similar reasons as before. 


\vspace{0.1in}
\subsection{Null-homologous exotic tori} \

In \cite{HS}, Hoffman and the second author of the current article showed that certain null-homologous tori arising in the knot surgery construction of Fintushel and Stern are topologically unknotted (i.e. bound a topologically embedded $S^1\times D^2$), but are smoothly knotted.  These null-homologous tori are found in a neighborhood of a homologically-essential torus, along which the knot surgery is performed. An example of one such exotically embedded trivial torus, and the steps for showing that it becomes smoothly trivial after a single stabilization, are given in Figure \ref{f:neiltorus}. 

\begin{figure}[htbp]
\labellist\hair 0pt
\pinlabel (a) at 94 129
\pinlabel (b) at 215 129
\pinlabel (c) at 331 129
\pinlabel (d) at 115 20
\pinlabel (e) at 236 20
\pinlabel (f) at 358 20
\endlabellist
\vspace{0.2in}
\includegraphics[height=65mm]{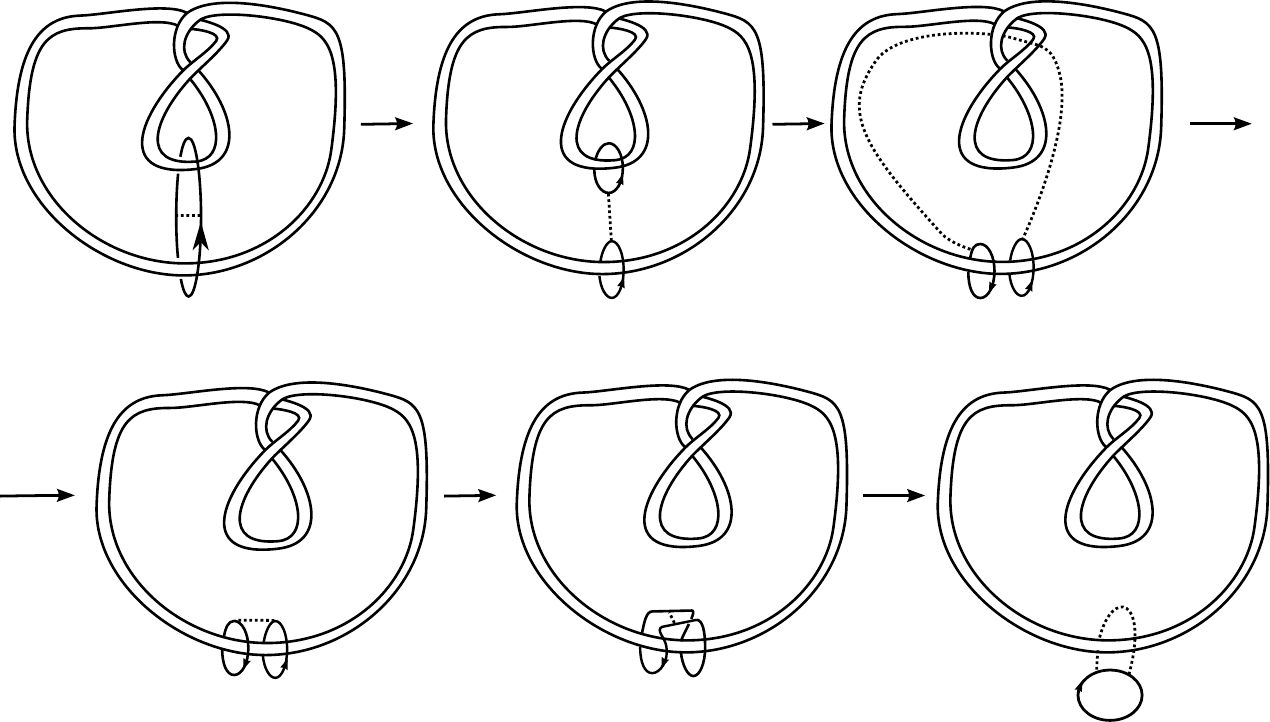}
\caption{Stabilizing null-homologous tori.} \label{f:neiltorus}
\end{figure}
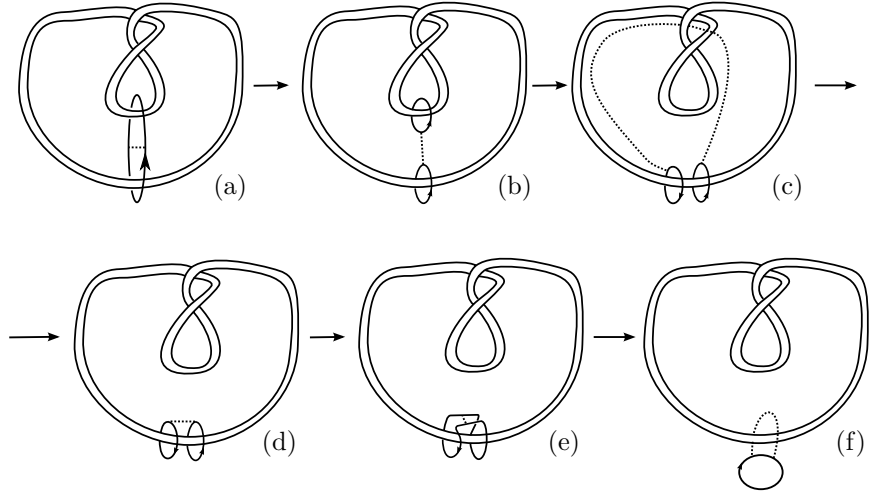
\vspace{0.2in}

A few more details about the construction are as follows. Let the $4$-manifold $X$ contain a square zero embedded torus $T$ with $\pi_1(X\setminus T) = 1$.  The  neighborhood of $T$, can be realized as $S^1 \times (S^3\setminus \, U)$, where $U$ is a neighborhood of the unknot. Figure \ref{f:neiltorus}a represents this neighborhood where we have supressed the $S^1$ factor. Specifically, we have depicted $S^3\setminus U$. The solid $S^1$ in the figure will correspond (after crossing with $S^1$) to a torus $\Sigma$  in the neighborhood of $T$. This $\Sigma$ is the exotically embedded torus in $X$ that we were looking for. We will show becomes smoothly trivial after a single stabilization. An infinite family of exotically embedded tori are constructed similarly in \cite{HS} and they can all be stabilized by following exactly the same pattern below.

The only feature of Figure \ref{f:neiltorus} that deserves further explanation is why the stabilizing handles represented by the dotted line in Figure \ref{f:neiltorus}c can be moved to it location in Figure \ref{f:neiltorus}d. This has to do with the fact that $\pi_1(X\setminus T) = 1$: the meridian of the unknot in Figure \ref{f:neiltorus}c is just the meridian of $T$ in $X$, and this meridian is null, homotopic in $X\setminus \nu T$. Therefore the arc denoting the handle can be moved ``across'' the complement of the unknot. Similarly, the handle in Figure \ref{f:neiltorus}f is trivial.

\begin{remark}
[Tori that bound $S^1 \times $ (punctured torus)] This example can be generalized as follows. Suppose $\Sigma$ is a null homologous torus. One stabilization is enough to make $\Sigma$ a trivially embedded surface (i.e. bounding a solid handlebody) under the following circumstances: (1) We have that $\Sigma$ bounds $S^1 \times T^{\circ}$, where $T^{\circ}$ is a punctured torus or Klein bottle; and (2) There exists a primitive loop in $pt \times T^{\circ}$ that bounds an immersed disk $D$ which intersects $S^1 \times T^{\circ}$ only along its boundary. When $T$ is a torus, we do an untwisted  stabilization, and when $T$ is a Klein bottle, a twisted one. This is a straightforward generalization of the steps in Figure \ref{f:neiltorus}.
\end{remark}

\vspace{0.1in}
\subsection{Knotted spheres of Auckly-Kim-Melvin-Ruberman} \

In \cite{AKMR}, Auckly, Kim, Melvin and Ruberman construct exotic surfaces through a method that does not follow the pattern above, i.e. their knotted surfaces do not arise via a $3$-dimensional surgery crossed with $S^1$. Their exotic surfaces arise as two different embeddings of the $2$-sphere $\mathbb{CP}^1 \subset X \# \CP$ for an almost completely decomposable $4$-manifold $X$. However, we can still show explicitly that these exotic surfaces become smoothly isotopic after adding a trivial handle. This is very different than the stabilization studied in \cite{AKMR}, where they show that the surfaces become smoothly isotopic in $X \# \CP \# S^2\times S^2$.

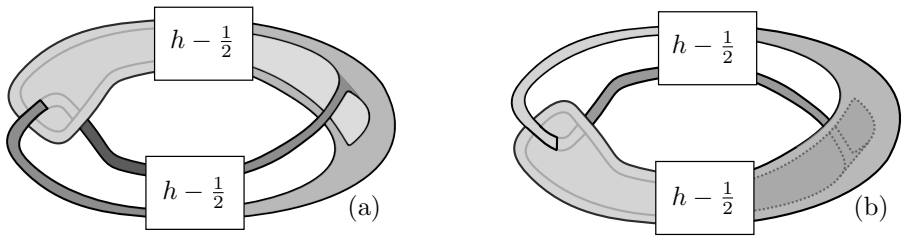
\begin{figure}[htbp]
\labellist\hair 0pt
\pinlabel $h-\frac{1}{2}$ at 74 73
\pinlabel $h-\frac{1}{2}$ at 71 16
\pinlabel $h-\frac{1}{2}$ at 263 71
\pinlabel $h-\frac{1}{2}$ at 263 14
\pinlabel (a) at 135 10
\pinlabel (b) at 326 10
\endlabellist
\vspace{0.2in}
\includegraphics{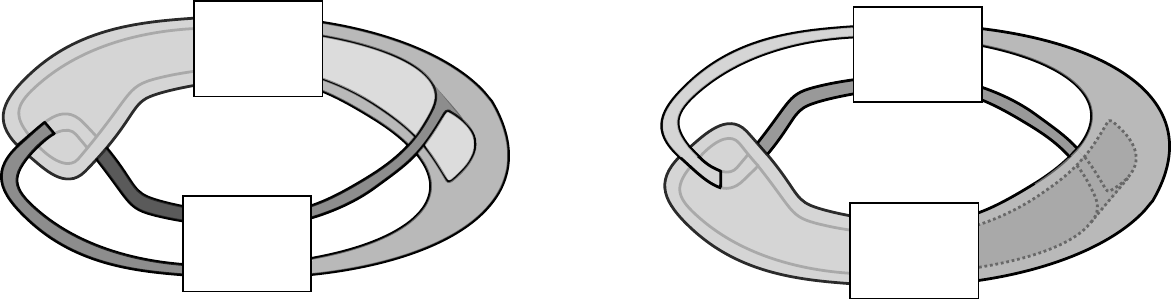}
\vspace{0.2in}
\caption{Pairs of exotic spheres in a small compact $4$-manifold. The $2$-handles are attached along the knot with $+1$ framing, and the boxes represent $h-\frac{1}{2}$ twists.}\label{f:stabruberman}
\end{figure}

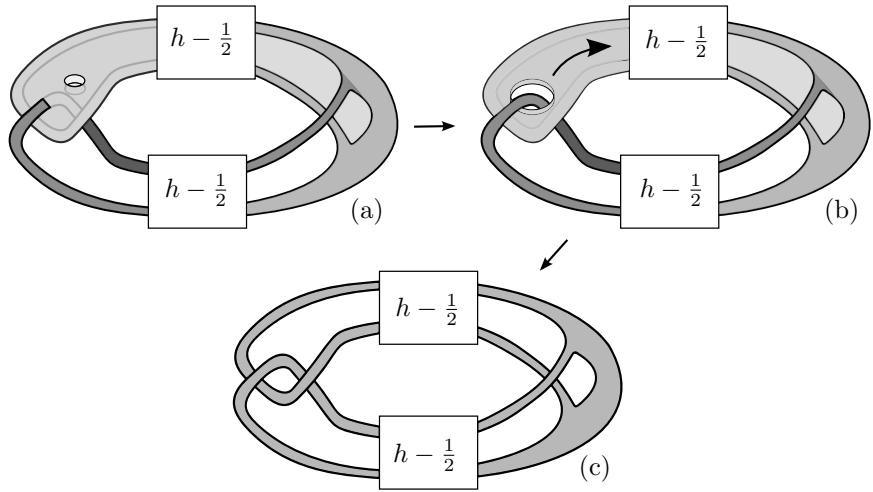
\begin{figure}[htbp]
\labellist\hair 0pt
\pinlabel $h-\frac{1}{2}$ at 73 170
\pinlabel $h-\frac{1}{2}$ at 71 113
\pinlabel $h-\frac{1}{2}$ at 252 169
\pinlabel $h-\frac{1}{2}$ at 250 113
\pinlabel $h-\frac{1}{2}$ at 158 68
\pinlabel $h-\frac{1}{2}$ at 157 13
\pinlabel (a) at 135 105
\pinlabel (b) at 314 105
\pinlabel (c) at 221 8
\endlabellist
\vspace{0.2in}
\includegraphics{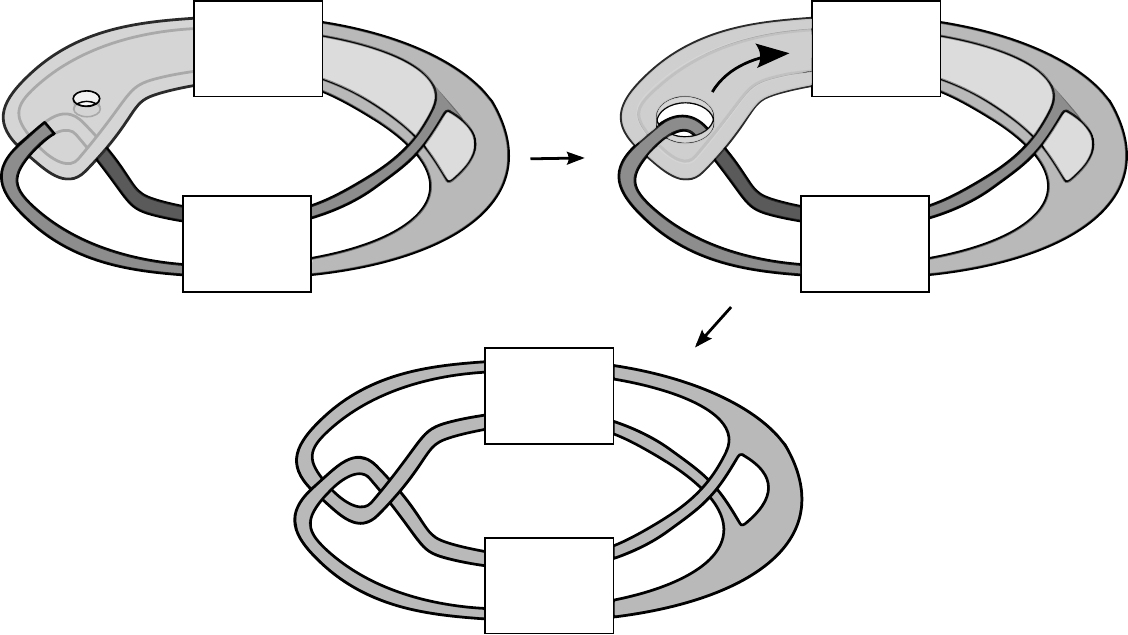}
\vspace{0.2in}
\caption{Stabilizing the spheres of Auckly-Kim-Melvin-Ruberman.}\label{f:stabruberman2}
\end{figure}

This construction can be considerably localized: The authors show that their surfaces are exotically embedded in a simple compact $4$-manifold composed of a \linebreak $0$-handle a $2$-handle. These compact $4$-manifolds, parametrized by $h$, are depicted in Figure \ref{f:stabruberman}, as are the pair of exotically embedded surfaces. One should understand the surfaces as being made up of the core of the 2-handle plus the depicted ribbon disk (i.e. push the depicted ribbon singularities down into the 0-handle).

These surfaces become isotopic after adding a trivial handle as illustrated in Figure \ref{f:stabruberman2}. One can see this handle explicitly as the tube added in Figure \ref{f:stabruberman2}(a) between two parts of the ribbon disk. The handle is trivial by Lemma \ref{l:mainLemma} because the surface has simply-connected complement. To pass from Figure \ref{f:stabruberman2}(a) to \ref{f:stabruberman2}(b), push the handle down into the 4-ball, move it over, and then bring it back up around the ribbon band. A further isotopy in the boundary of the 0-handle takes us from (b) to (c). Performing a similar stabilization on the surface in Figure \ref{f:stabruberman}(b) has the same result, and consequently the surfaces in Figures \ref{f:stabruberman}(a) and (b) become smoothly isotopic after adding a single trivial handle.

\smallskip
Our arguments can certainly be applied to other possible families of exotic knottings arising from exotic ACD $4$-manifolds containing similar submanifolds as the one in Figure~\ref{f:stabruberman}.

\vspace{0.3in}

\end{document}